\documentclass[12pt]{amsart}
\usepackage{amssymb}
\usepackage{amsfonts}
\usepackage{graphicx}
\usepackage{amscd}

\newtheorem{theorem}{Theorem}

\newtheorem{corollary}{Corollary}

\numberwithin{equation}{section}

\begin{document}
\title[New inequalities involving ...]{Some new inequalities involving  Heinz operator means }
\author[Ghazanfari et al.]{ A. G. Ghazanfari$^{1,*}$, S. Malekinejad$^{2}$, S. Talebi$^{3}$}
\thanks{*Corresponding author.\\
E-mail addresses:\\
ghazanfari.a@lu.ac.ir(A.G. Ghazanfari), maleki60313@gmail.com(S. Malekinejad), talebi$_{-}$s@pnu.ac.ir(S. Talebi)
 }

   \address{}
\email{} \maketitle {\small  \centerline{ \em $^{1}$Department of
Mathematics, Lorestan University}
\centerline{ \em P.O.Box 465, Khoramabad, Iran}

 \address{}
\email{} \maketitle {\small  \centerline{ \em $^{2}$Department of
Mathematics, Payame Noor University}
\centerline{ \em P.O.Box 19395-3697, Tehran, Iran}

 \address{}
\email{} \maketitle {\small  \centerline{ \em $^{3}$Department of
Mathematics, Payame Noor University}
\centerline{ \em P.O.Box 91735-433, Mashhad, Iran}

\begin{abstract}\noindent
We give some new refinements of Heinz inequality and an improvement of the reverse Young's inequality for
scalars and we use them to establish new inequalities for operators and the Hilbert -Schmidt norm of matrices.
We give a uniformly and abbreviated form of the inequalities presented by Kittaneh and Mansarah, and the inequalities presented by Kai
and we obtain some of their operator and matrix versions.
\end{abstract}
\maketitle

\section{introduction}\vspace{.2cm} \noindent

The classical Young inequality says that if $a,b\geq 0$ and $\nu\in[0,1]$, then
\begin{equation}\label{1.1}
a^{\nu}b^{1-\nu}\leq \nu a+(1-\nu) b
\end{equation}
with equality if and only if $a=b$. Young$^,$s inequality for scalars is not only interesting in itself but also very useful.
If $\nu=\frac{1}{2}$, by (\ref {1.1}), we obtain the arithmetic -geometric mean inequality
\begin{equation}\label{1.2}
2\sqrt{ab}\leq a+b
\end{equation}

Kittaneh and Mansarah \cite{kit2, kit3} obtained a refinement of Young$^,$s inequality and its reverse as follows:

\begin{equation}\label{1.3}
a^{\nu}b^{1-\nu}+r_0\left(\sqrt{a}-\sqrt{b}\right)^2\leq \nu a+(1-\nu) b,
\end{equation}
where $r_0=min\{\nu,1-\nu\}$.\\

\begin{equation}\label{1.4}
\nu a+(1-\nu) b\leq a^{\nu}b^{1-\nu}+R_0\left(\sqrt{a}-\sqrt{b}\right)^2
\end{equation}
where $R_0=max\{\nu,1-\nu\}$.

Zhao and Wu in \cite{zha2} obtained refinements of (\ref{1.3}) and (\ref{1.4}):\\
If $0<\nu\leq \frac{1}{2}$, then

\begin{align}\label{1.5}
&a^{1-\nu}b^\nu+\nu(\sqrt{a}-\sqrt{b})^2+r_0(\sqrt[4]{ab}-\sqrt{a})^2\leq (1-\nu)a+\nu b\notag\\
&\leq a^{1-\nu}b^\nu+(1-\nu)(\sqrt{a}-\sqrt{b})^2-r_0(\sqrt[4]{ab}-\sqrt{b})^2,
\end{align}
If $\frac{1}{2}<\nu<1$, then
\begin{align}\label{1.6}
&a^{1-\nu}b^\nu+(1-\nu)(\sqrt{a}-\sqrt{b})^2+r_0(\sqrt[4]{ab}-\sqrt{b})^2\leq (1-\nu)a+\nu b\notag\\
&\leq a^{1-\nu}b^\nu+\nu(\sqrt{a}-\sqrt{b})^2-r_0(\sqrt[4]{ab}-\sqrt{a})^2
\end{align}

In the recent paper \cite{sab}, the authors present multiple-term refiements of Young's inequality and its reverse as follows:

\begin{align}\label{1.7}
a^\nu b^{1-\nu} +\sum_{j=0}^{N-1} s_j\left(\sqrt{a}^{\nu_j+1/2^j}\sqrt{b}^{1-\nu_j-1/2^j}-\sqrt{a}^{\nu_j}\sqrt{b}^{1-\nu_j}\right)^2\leq\nu a+(1-\nu)b,
\end{align}

\begin{align}\label{1.8}
(1+\nu)a-\nu b+\nu\sum_{j=1}^N 2^{j-1}\left(\sqrt{a}-\sqrt[2^j]{a^{2^{j-1}-1}b}\right)^2\leq a^{1+\nu}b^{-\nu}.
\end{align}

So far the mentioned inequalities were involving a convex combination of $a, b$, i.e., $a\nabla_\nu b=\nu a+(1-\nu)b$.

There exist some inequalities which involving a nonconvex combination of $a, b$. In the following some of them are listed.

 Kai in \cite{kai} gave the following Young type inequalities

\begin{equation}\label{1.9}
(\nu^2 a)^{\nu}b^{1-\nu}+\nu^2\left(\sqrt{a}-\sqrt{b}\right)^2\leq \nu^2 a+(1-\nu)^2 b, ~0\leq \nu\leq \frac{1}{2}
\end{equation}
and
\begin{equation}\label{1.10}
a^\nu\left((1-\nu)^2b\right)^{1-\nu}+(1-\nu)^2\left(\sqrt{a}-\sqrt{b}\right)^2\leq \nu^2 a+(1-\nu)^2 b, ~\frac{1}{2}\leq \nu\leq 1.
\end{equation}

Recently, Burqan and Khandaqji \cite{bur} gave the following reverses of the scalar Young type inequality

\begin{equation}\label{1.11}
\nu^2a+(1-\nu)^2b\leq \nu^2(\sqrt{a}-\sqrt{b})^2+(a\nu^2)^{\nu}b^{1-\nu},
\end{equation}
for $a,b\geq0$ and $\nu\in[\frac{1}{2},1]$, and

\begin{equation}\label{1.12}
\nu^2a+(1-\nu)^2b\leq (1-\nu)^2(\sqrt{a}-\sqrt{b})^2+a^\nu(1-\nu)^{2-2\nu}b^{1-\nu}.
\end{equation}
for $a,b\geq0$ and $\nu\in[0,\frac{1}{2}]$.

Also, the following important inequalities were obtained by Cartwright and Field \cite{car},
\begin{equation}\label{1.13}
a^\nu b^{1-\nu}+\frac{\nu(1-\nu)}{2M}(a-b)^2\leq \nu a+(1-\nu) b\leq a^\nu b^{1-\nu}+\frac{\nu(1-\nu)}{2m}(a-b)^2,
\end{equation}
where $a>0,b>0,~ m=min\{a,b\}, ~ M=max\{a,b\}$ and $0\leq\nu\leq 1$.\\

Let $a,b\geq0$ and $ 0\leq \nu\leq1$. The Heinz mean and Heron mean respectively are defined as follows:
\begin{equation*}
H_{\nu}(a,b)=\frac{a^{\nu}b^{1-\nu}+a^{1-\nu}b^{\nu}}{2},
\end{equation*}
and
\begin{equation*}
F_{\nu}(a,b)=(1-\nu)\sqrt{ab}+\nu \frac{a+b}{2}.
\end{equation*}

It follows from the inequalities (\ref{1.1}) and (\ref{1.2}) that the Heinz and Heron means interpolate between the geometric mean and arithmetic mean:
\begin{equation}\label{1.14}
\sqrt{ab}\leq H_{\nu}(a,b)\leq \frac{a+b}{2},
\end{equation}
and
\begin{equation}\label{1.15}
\sqrt{ab}\leq F_{\nu}(a,b)\leq \frac{a+b}{2}.
\end{equation}

The second inequality of (\ref{1.14}) is know as Heinz inequality for nonnegative real numbers.\\
Bhatia \cite{bha}, proved that the Heinz and the Heron means
satisfy the following inequality
\begin{align*}
H_\nu(a,b)\leq F_{\alpha(\nu)}(a,b),
\end{align*}
where $\alpha(\nu)=1-4\left(\nu-\nu^2\right)$.

As a direct consequence of the inequality (\ref{1.3}), Kittaneh and Manasrah \cite{kit3} obtained a refinement of  the Heinz inequality as follows:
\begin{equation}\label{1.16}
 H_{\nu}(a,b)+r_0\left(\sqrt{a}-\sqrt{b}\right)^2\leq  \frac{a+b}{2},
\end{equation}
where $r_0=min\{\nu,1-\nu\}$.

Zou and Jiang \cite{zou} obtained a refinement of inequality (\ref{1.16}) as follows:

\begin{theorem}\label{t1}
Let $a, b \geq 0$ and $0 \leq \nu \leq 1$. If $r_0 = min \{\nu, 1 - \nu\}$, and suppose
that
\begin{equation*}
\phi(\nu)=\frac{a^{\nu}b^{1-\nu}+a^{1-\nu}b^{\nu}}{2}
\end{equation*}
for $\nu\in[0,1]$.
Then

 \begin{align}\label{1.17}
\phi(\nu)\leq
\begin{cases}
(1-4r_0)\phi(0)+4r_0\phi(\frac{1}{4}) & \text{ \( \nu\in[0,\frac{1}{4}]\cup[\frac{3}{4},1] \),}\\
(4r_0-1)\phi(\frac{1}{2})+2(1-2r_0)\phi(\frac{1}{4}) & \text{ \( \nu\in[\frac{1}{4},\frac{3}{4}] \)}.
\end{cases}
\end{align}
\end{theorem}

\section{The results and discussion}

First, we give some new refinements of Heinz inequality and an improvement of the reverse Young's inequality for
scalars and we use them to establish new inequalities for operators and the Hilbert -Schmidt norm of matrices.

\begin{theorem}\label{t2}
Let $a,b\geq0$ and $0\leq\nu\leq1$ then
\begin{equation}\label{2.1}
(1-\nu^2+\nu^3)a+(1-\nu^2)b\leq \nu^{\nu-2}a^\nu b^{1-\nu}+(\sqrt{a}-\sqrt{b})^2.
\end{equation}
\end{theorem}

\begin{proof}
\begin{align*}
&\nu^{\nu-2}a^\nu b^{1-\nu}+(\sqrt{a}-\sqrt{b})^2-(1-\nu^2+\nu^3)a-(1-\nu^2)b\\
&=\nu^{\nu-2}a^\nu b^{1-\nu}+(a+b-2\sqrt{ab})-(1-\nu^2+\nu^3)a-(1-\nu^2)b\\
&=\nu^{\nu-2}a^\nu b^{1-\nu}+\nu^2(1-\nu)a+\nu^2b-2\sqrt{ab}\\
&\geq\nu^{\nu-2}a^\nu b^{1-\nu}+(\nu^2a)^{1-\nu}(\nu b)^\nu-2\sqrt{ab}\\
&=\nu^{\nu-2}a^\nu b^{1-\nu}+\nu^{2-\nu}a^{1-\nu}b^\nu-2\sqrt{ab}\\
&=\left(\nu^{\frac{\nu}{2}-1}a^{\frac{\nu}{2}} b^{\frac{1-\nu}{2}}-\nu^{1-\frac{\nu}{2}}a^{\frac{1-\nu}{2}} b^{\frac{\nu}{2}}\right)^2\geq0.
\end{align*}
\end{proof}

In view of the inequalities (\ref{2.1}) and second inequalities in (\ref{1.5}) and (\ref{1.6}), we want to know the relationship between them.
It is easy to observe that the left hand side and the right hand
side in the inequality (\ref{2.1}) are greater than or equal to the corresponding sides in the inequalities (\ref{1.5}) and (\ref{1.6}), respectively. It
should be noticed here that neither (\ref{2.1}) nor second inequalities in (\ref{1.5}) and (\ref{1.6}) is uniformly better than the other.

\subsection{operator inequalities}

Let $H$ be a Hilbert space and let $B_h(H)$ be the semi-space of all bounded linear self-adjoint operators on $H$.
Further, let $B(H)$ and $B(H)^+$, respectively, denote the  set of all bounded linear operators on a complex
Hilbert space $H$ and set of all positive operators in $B_h(H)$. The set of all positive invertible operators
is denoted by $B(H)^{++}$. For $A,B\in B(H)$, $A^*$ denotes the conjugate operator of $A$. An operator $A\in B(H)$
is positive, and we write $A\geq 0$, if $(Ax,x)\geq0$ for every vector $x\in H$. If $A$ and $B$ are self-adjoint
operators, the order relation $A\geq B$ means, as usual, that $A-B$ is a positive operator. The theory of operator
 means for positive (bounded linear) operators on a Hilbert space was initiated by T. Ando and established by him
 and F. Kubo in connection with L\"{o}wners theory for the operator monotone functions \cite {Kub}.

Let $M_n(\mathbb{C})$ be the algebra of $n\times n$ complex matrices. The Hilbert-Schmidt (or frobenius) norm
of $A=[a_{ij}]\in M_n(\mathbb{C})$ is denoted by
\begin{equation*}
\|A\|_2=(\sum_{j=1}^n s_j^2(A))^{\frac{1}{2}},
\end{equation*}
 where $s_1(A)\geq s_2(A)\geq...\geq s_n(A)$ are the singular values of $A$, which are the eigenvalues of the positive semidefinite
 matrix $\mid A\mid =(AA^*)^{\frac{1}{2}}$, arranged in decreasing order and repeated according to multiplicity. It is known that
 the Hilbert-Scmidt norm is unitarily invariant and if $A=[a_{ij}]\in M_n(\mathbb{C})$, then $\|A\|_2=\left(\sum_{i,j=1}^n |a_{ij}|^2\right)^{\frac{1}{2}}$.

For Hermitian matrices $A,B\in M_n(\mathbb{C})$, we write that $A\geqslant  0$
if $A$ is positive semidefinite, $A>0$ if $A$ is positive definite, and $A\geqslant B$ if $A-B\geqslant 0$.

Let $A,B\in B(H)$ be two positive operators and $\nu\in [0,1]$, then $\nu$-weighted arithmetic mean of $A$ and $B$ denoted by $A\nabla_{\nu} B$, is defined as $A\nabla_{\nu} B=(1-\nu)A+\nu B$. If $A$ is invertible, the $\nu$-geometric mean of $A$ and $B$ denoted by $A\sharp_{\nu} B$ is defined as $A\sharp_{\nu} B=A^{\frac{1}{2}}(A^{\frac{-1}{2}}BA^{\frac{-1}{2}})^{\nu} A^{\frac{1}{2}}$. In addition if both $A$ and $B$ are invertible, the $\nu$-harmonic mean of $A$ and $B$, denoted by $A!_{\nu} B$ is defined as $A!_{\nu} B=((1-v)A^{-1}+vB^{-1})^{-1}$. For more detail, see F. Kubo and T. Ando \cite {Kub}. When $v=\frac{1}{2}$ , we write $A\nabla B$, $A\sharp B$, $A!B$ for brevity, respectively. The operator version of the Heinz means, is defined by
\begin{equation*}
 H_{\nu}(A,B)=\frac{A\sharp_{\nu} B+A\sharp_{1-\nu} B}{2}
 \end{equation*}
 where $A,B\in B(H)^{++}$, and $\nu\in [0,1]$. The operator version of the Heron means, is defined by
\begin{equation*}
 F_\alpha(A,B)=(1-\alpha)(A\sharp B)+\alpha( A\nabla B)
 \end{equation*}
  for $0\leq \alpha \leq1$.

To reach inequalities for bounded self-adjoint operators on Hilbert space, we shall use the following monotonicity property for operator functions:\\
If $X\in B_h(H)$ with a spectrum $Sp(X)$ and $f,g$  are continuous real-valued functions on an interval containing $Sp(X)$, then
\begin{equation}\label{2.2}
f(t)\geq g(t),~t\in Sp(X)\Rightarrow ~f(X)\geq g(X).
\end{equation}
For more details about this property, the reader is referred to \cite {pec}.

Zhao et al. in \cite{zha} gave an inequality
for the Heron mean as follows:\\
If $A$ and $B$ be two positive and invertible operators then
$$H_{\nu}(A,B)\leq F_{\alpha(\nu)}(A,B)$$
for $\nu\in [0,1]$, where $\alpha(\nu)=1-4(\nu-\nu^2)$.\\
\begin{theorem}\label{t3}
Let  $A,B\in B(H)^{++}$. If $\nu\in[0,1]$, then
\begin{equation}\label{2.3}
\nu^2(\nu-2)A\nabla B+2(A\sharp B)\leq\nu^{\nu-2}H_{\nu}(A,B).
\end{equation}
\end{theorem}
\begin{proof}
If $\nu\in[0,1]$, the inequality (\ref{2.1}) for $a=1,~b>0$, becomes
\begin{equation*}
(1-\nu^2+\nu^3)+(1-\nu^2)b\leq \nu^{\nu-2} b^{1-\nu}+(1+b-2\sqrt{b}).
\end{equation*}
The operator $X=A^{-\frac{1}{2}} B A^{-\frac{1}{2}}$ has a positive spectrum. According to (\ref{2.2}), we can
insert $A^{-\frac{1}{2}} B A^{-\frac{1}{2}}$
in above inequality, i.e., we have
\begin{align}\label{2.4}
&(1-\nu^2+\nu^3)1+(1-\nu^2)A^{-\frac{1}{2}} B A^{-\frac{1}{2}}\\
&\leq \nu^{\nu-2}( A^{-\frac{1}{2}} B A^{-\frac{1}{2}})^{1-\nu}+(1+A^{-\frac{1}{2}} B A^{-\frac{1}{2}}-2(A^{-\frac{1}{2}} B A^{-\frac{1}{2}})^{\frac{1}{2}}\notag
\end{align}
Finally, if we multiply inequality (\ref{2.4}) by $A^{\frac{1}{2}}$ on the left and right sides, we get
\begin{equation}\label{2.5}
(1-\nu^2+\nu^3)A+(1-\nu^2)B\leq\nu^{\nu-2}(A\sharp_{1-\nu}B)+(A+B-2(A\sharp B)).
\end{equation}
It is shown in \cite[p.148]{pec}, that $B\sharp_{1-\nu}A=A\sharp_\nu B$ . By replacing $A$ by $B$ and $B$ by $A$, we have
\begin{equation}\label{2.6}
(1-\nu^2+\nu^3)B+(1-\nu^2)A\leq\nu^{\nu-2}(A\sharp_{\nu}B)+(A+B-2(A\sharp B)).
\end{equation}
 Adding (\ref{2.5}) and (\ref{2.6}), we have
\begin{equation*}
\nu^2(\nu-2)A\nabla B+2(A\sharp B)\leq\nu^{\nu-2}H_{\nu}(A,B).
\end{equation*}
\end{proof}

\begin{theorem}\label{t4}
Let $A,B$ be invertible positive operators in $B(H)$ and $A\leq B$. If $0\leq\nu\leq 1$, then
\begin{align}\label{2.7}
H_\nu(A,B)+\frac{\nu(1-\nu)}{2}(B-2A+AB^{-1}A)&\leq A\nabla B\\
\leq H_\nu(A,B)+\frac{\nu(1-\nu)}{2}(A-2B+BA^{-1}B).\notag
\end{align}
\end{theorem}

\begin{proof}
From (\ref{1.13}) we obtain the following inequalities
\begin{align}\label{2.8}
H_\nu(a,b)+\frac{\nu(1-\nu)}{2M}(a-b)^2\leq a\nabla b\leq H_\nu(a,b)+\frac{\nu(1-\nu)}{2m}(a-b)^2.
\end{align}
The first inequality in (\ref{2.8}) for $a=1\leq b$, becomes
\begin{equation*}
\frac{b^\nu+b^{1-\nu}}{2}+\frac{\nu(1-\nu)}{2}(1-b)^2b^{-1}\leq \frac{1+b}{2}
\end{equation*}
The operator $X=A^{-\frac{1}{2}} B A^{-\frac{1}{2}}\geq I$, since $A\leq B$. According to (\ref{2.2}), we can
insert $X$
in above inequality, i.e., we have
\begin{align}\label{2.9}
\frac{X^\nu+X^{1-\nu}}{2}+\frac{\nu(1-\nu)}{2}\big(X-2I+X^{-1}\big)
\leq \frac{I+X}{2}
\end{align}
Finally, if we multiply inequality (\ref{2.9}) by $ A^{-\frac{1}{2}}$ on the left and right sides, we get the first inequality in (\ref{2.7}).
Similarly, the second inequality in (\ref{2.7}) is obtained.
\end{proof}
It is known that if $t$ is a positive real number then $t+\frac{1}{t}\geq 2$,
therefore if $X$ is a invertible positive operator in $B(H)$, then $2I\leq X+X^{-1}$.
Therefore $2I\leq A^{\frac{-1}{2}}BA^{\frac{-1}{2}}+A^{\frac{1}{2}}B^{-1}A^{\frac{1}{2}}$ or $2A\leq B+AB^{-1}A$.
So $AB^{-1}A+B-2A\geq0$, this implies that the inequality (\ref{2.7}) is a refinement of the operator Heinz inequality.

\begin{theorem}\label{t5}
Let $A,B$ be positive operators in $B(H)$. If $0\leq\nu\leq 1$, then
\begin{align}\label{2.10}
r^{2r}H_{\nu}(A,B)+(2r-1)(A\nabla B)&\leq2r^2(A\sharp B)\leq 2R^2(A\sharp B)\\
&\leq R^{2R}H_{\nu}(A,B)+(2R-1)(A\nabla B).\notag
\end{align}
Where $r=min\{\nu, 1-\nu\},~ R=max\{\nu, 1-\nu\}$.
\end{theorem}

\begin{proof}
From inequlities  (\ref{1.9}),  (\ref{1.10}),  (\ref{1.11}) and  (\ref{1.12}), we obtain that if $0\leq\nu\leq1$ then

\begin{align}\label{2.11}
r^{2r}a^\nu b^{1-\nu} +r^2(\sqrt{a}-\sqrt{b})^2&\leq \nu^2a+(1-\nu)^2 b\\
&\leq R^{2R}a^\nu b^{1-\nu} +R^2(\sqrt{a}-\sqrt{b})^2,\notag
\end{align}
where $r=min\{\nu, 1-\nu\},~ R=max\{\nu, 1-\nu\}$.
This implies that
\begin{align}\label{2.12}
r^{2r}H_\nu(a,b)+(2r-1)a\nabla b&\leq 2r^2\sqrt{ab}\leq 2R^2\sqrt{ab}\\
&\leq R^{2R}H_\nu (a,b)+(2R-1)a\nabla b.\notag
\end{align}

 From inequality (\ref{2.12}), by the same method used in the proof of Theorem \ref{t4},
we get the inequalities in (\ref{2.10}).
\end{proof}


\subsection{Heinz and Young type inequalities and their reverses for matrices }

\begin{theorem}\label{t6}
Let $A,B,X \in M_n(\mathbb{C})$ such that $A$ and $B$ are positive definite if $\nu\in[0,1]$, then
\begin{align}\label{2.13}
&\left\|\nu^2(\nu-2)(AX+X B)\right\|_2\notag\\
&\leq\left\|\nu^{\nu-2}(A^{\nu}X B^{1-\nu}+A^{1-\nu}XB^{\nu})-4(A^{\frac{1}{2}}XB^{\frac{1}{2}})\right\|_2\\
&\leq\nu^{\nu-2}\left\|A^{\nu}X B^{1-\nu}+A^{1-\nu}XB^{\nu}\right\|_2+4\left\|(A^{\frac{1}{2}}XB^{\frac{1}{2}})\right\|_2.\notag
\end{align}

\end{theorem}
\begin{proof}
Since $A$ and $B$ are positive semidefinite, it follows by the spectral theorem that there exist unitary matrices $U,V\in M_n(\mathbb{C})$ such that
\begin{equation*}
A=U\Gamma_1 U^*~~and ~~B=V\Gamma_2V^*,
\end{equation*}
where
\begin{equation*}
\Gamma_1=diag(\lambda_1,...,\lambda_n),~~~~~\Gamma_2=diag(\mu_1,...,\mu_n),~~~\lambda_i,\mu_i\geq0,~~i=1,...,n
\end{equation*}
Let
\begin{equation*}
Y=U^*XV=[y_{ij}],
\end{equation*}
then
\begin{align*}
&\frac{A^{\nu}X B^{1-\nu}+A^{1-\nu}XB^{\nu}}{2}\\
&=\frac{(U\Gamma_1U^*)^{\nu}X(V\Gamma_2V^*)^{1-\nu}+(U\Gamma_1U^*)^{1-\nu}X(V\Gamma_2V^*)^{\nu}}{2}\\
&=\frac{(U\Gamma_1^{\nu}U^*)X(V\Gamma_2^{1-\nu}V^*)+(U\Gamma_1^{1-\nu}U^*)X(V\Gamma_2^{\nu}V^*)}{2}\\
&=\frac{U\Gamma_1^{\nu}(U^*XV)\Gamma_2^{1-\nu}V^*+U\Gamma_1^{1-\nu}(U^*XV)\Gamma_2^{\nu}V^*}{2}\\
&=U\left(\frac{\Gamma_1^{\nu}Y\Gamma_2^{1-\nu}+\Gamma_1^{1-\nu}Y\Gamma_2^{\nu}}{2}\right )V^*.
\end{align*}
Therefore,
\begin{align*}
\left\|\frac{A^{\nu}X B^{1-\nu}+A^{1-\nu}XB^{\nu}}{2}\right\|^2_2&=
\left\|\frac{\Gamma_1^{\nu}Y\Gamma_2^{1-\nu}+\Gamma_1^{1-\nu}Y\Gamma_2^{\nu}}{2}\right\|_2^2\\
&=\sum_{i,j=1}^n\left(\frac{\lambda_i^{\nu}\mu_j^{1-\nu}+\lambda_i^{1-\nu}\mu_j^{\nu}}{2}\right)^2|y_{ij}|^2.
\end{align*}
Similarly, we have
\begin{align*}
&\sum_{i,j=1}^n\left(\sqrt{\lambda_i\mu_j}\right)^2|y_{ij}|^2=\left\| A^{\frac{1}{2}}XB^{\frac{1}{2}}\right\|^2_2,\\
&\sum_{i,j=1}^n\left(\frac{\lambda_i+\mu_j}{2}\right)^2|y_{ij}|^2=\left\| \frac{ AX+X B}{2}\right\|^2_2.
\end{align*}
It follows from the inequality (\ref{2.1}) that
\begin{align*}
&\left\|\nu^2(\nu-2)(AX+X B)\right\|_2^2
=\sum_{i,j=1}^n\left(\nu^2(\nu-2)\frac{\lambda_i+\mu_j}{2}\right)^2|y_{ij}|^2\\
&\leq\sum_{i,j=1}^n\left(\nu^{\nu-2}\frac{\lambda_i^{\nu}\mu_j^{1-\nu}+\lambda_i^{1-\nu}\mu_j^{\nu}}{2}-4\sqrt{\lambda_i\mu_j}\right)^2|y_{ij}|^2\\
&=\left\|\nu^{\nu-2}A^{\nu}X B^{1-\nu}+A^{1-\nu}XB^{\nu}-4(  A^{\frac{1}{2}}XB^{\frac{1}{2}})\right\|_2^2.\notag
\end{align*}
Therefore
\begin{align*}
&\left\|\nu^2(\nu-2)(AX+X B)\right\|_2\\
&\leq\left\|\nu^{\nu-2}A^{\nu}X B^{1-\nu}+A^{1-\nu}XB^{\nu}-4(  A^{\frac{1}{2}}XB^{\frac{1}{2}})\right\|_2\\
&\leq\nu^{\nu-2}\left\|A^{\nu}X B^{1-\nu}+A^{1-\nu}XB^{\nu}\right\|_2+4\left\|(  A^{\frac{1}{2}}XB^{\frac{1}{2}})\right\|_2.\notag
\end{align*}
\end{proof}

\begin{theorem}\label{t7}
Let $A,B,X \in M_n(\mathbb{C})$ such that $A, B$ and $X$ are positive definite. If $\nu\in[0,1]$, then
\begin{align}\label{2.14}
&\left\|A^{\nu}X B^{1-\nu}+A^{1-\nu}XB^{\nu}+\nu(1-\nu)\alpha\left(A^2X+XB^2-2AXB\right)\right\|_2\\
&\leq\left\|AX+XB\right\|_2,\notag
\end{align}
where $\alpha=min\{\|A\|^{-1}, \|B\|^{-1}\}$.
\end{theorem}

\begin{proof}
Since $A$ and $B$ are positive definite, we have $\|A\|^{-1}= \inf\{|\lambda|^{-1}:\lambda\in sp(A)\}$ and
 $\|B\|^{-1}= \inf\{|\mu|^{-1}:\lambda\in sp(B)\}$. It follows by the spectral theorem that there exist unitary matrices $U,V\in M_n(\mathbb{C})$ such that
\begin{equation*}
A=U\Gamma_1 U^*~~and ~~B=V\Gamma_2V^*,
\end{equation*}
where
\begin{equation*}
\Gamma_1=diag(\lambda_1,...,\lambda_n),~~~~~\Gamma_2=diag(\mu_1,...,\mu_n),~~~\lambda_i,\mu_i> 0,~~i=1,...,n
\end{equation*}
Let
\begin{equation*}
Y=U^*XV=[y_{ij}],
\end{equation*}
then
\begin{align*}
&A^{\nu}X B^{1-\nu}+A^{1-\nu}XB^{\nu}+\nu(1-\nu)\alpha\left(A^2X+XB^2-2AXB\right)\\
&=U\left(\Gamma_1^\nu Y\Gamma_2^{1-\nu}+\Gamma_1^{1-\nu} Y\Gamma_2^{\nu}+\nu(1-\nu)\alpha\left(\Gamma_1^2Y+Y\Gamma_2^2-2\Gamma_1Y\Gamma_2\right)\right)V^*,
\end{align*}
and
\begin{align*}
AX+XB=U(\Gamma_1Y+Y\Gamma_2)V^*.
\end{align*}
From the first inequality in (\ref{2.8}) we conclude that
\begin{align*}
\lambda_i^{\nu}\mu_j^{1-\nu}+\lambda_i^{1-\nu}\mu_j^{\nu}+\nu(1-\nu)(\lambda_i-\mu_j)^2\alpha\leq \lambda_i+\mu_j.
\end{align*}
The rest of the proof is similar to the proof of Theorem \ref{t6}. However, the details are omitted.
\end{proof}
In the following, we show that the inequality (\ref{2.14}) is a refinement of the matrix Heinz inequality.

\begin{corollary}
Let $A,B,X \in M_n(\mathbb{C})$ such that $A, B$ and $X$ are positive definite. If $\nu\in[0,1]$, then
\begin{align*}
&\left\|A^{\nu}X B^{1-\nu}+A^{1-\nu}XB^{\nu}\right\|_2\\
&\leq\left[\left\|A^{\nu}X B^{1-\nu}+A^{1-\nu}XB^{\nu}\right\|_2^2+\nu^2(1-\nu)^2\alpha^2\left\|\left(A^2X+XB^2-2AXB\right)\right\|_2^2\right]^\frac{1}{2}\\
&\leq\left\|A^{\nu}X B^{1-\nu}+A^{1-\nu}XB^{\nu}+\nu(1-\nu)\alpha\left(A^2X+XB^2-2AXB\right)\right\|_2\\
&\leq\left\|AX+XB\right\|_2,
\end{align*}
where $\alpha=min\{\|A\|^{-1}, \|B\|^{-1}\}$.

\end{corollary}

\begin{proof}
The first inequality is trivial. For the second inequality, we note that if $a$ and $\lambda_i, \mu_j~(i,j=1,2,...,n)$, are positive real numbers, then
\[
\sum_{i,j=1}^n(\lambda_i^2+a^2\mu_j^2)\leq\sum_{i,j=1}^n(\lambda_i+a\mu_j)^2.
\]
\end{proof}

Using the same strategy as in the proof of Theorem \ref{t6} and inequality (\ref{2.12}), we get the following theorem:

\begin{theorem}\label{t8}
Let $A,B,X \in M_n(\mathbb{C})$ such that $A$ and $B$ are positive semidefinite if $0\leq\nu\leq 1$, then
\begin{align*}
& r^{2r}\left\|A^{\nu}X B^{1-\nu}+A^{1-\nu}XB^{\nu}+(2r-1)(AX+X B)\right\|_2\\
&\leq2r^2\left\|  A^{\frac{1}{2}}XB^{\frac{1}{2}}\right\|_2\leq 2R^2\left\|  A^{\frac{1}{2}}XB^{\frac{1}{2}}\right\|_2\\
&\leq R^{2R}\left\|A^{\nu}X B^{1-\nu}+A^{1-\nu}XB^{\nu}+(2R-1)(AX+X B)\right\|_2.\notag
\end{align*}
Where $r=min\{\nu, 1-\nu\},~ R=max\{\nu, 1-\nu\}$.
\end{theorem}

\section{Conclusions}

In the present paper we got some improved Heinz type inequalities for
operators and the Hilbert -Schmidt norm of matrices. Meanwhile, We gave
a uniformly and abbreviated form of the inequalities presented by
Kittaneh and Mansarah, and the inequalities presented by Kai and
we obtained some of their operator and matrix versions.

{ Acknowledgements}\\
The authors wish to express their hearty thanks to the referees for their valuable comments, detailed corrections, and
suggestions for revising the manuscript.

\end{document}